\documentclass[12pt, reqno]{amsart}
\usepackage[bookmarksnumbered,colorlinks, plainpages]{hyperref}
\hypersetup{colorlinks=true,linkcolor=red, anchorcolor=green, citecolor=cyan, urlcolor=red, filecolor=magenta, pdftoolbar=true}
\usepackage{fixltx2e}

\RequirePackage{tkz-berge}
\usepackage{amsmath, amsthm, amscd, amsfonts, amssymb, graphicx, color, enumerate, xcolor,  mathrsfs, latexsym}

 \definecolor{cupgreen}{rgb}{0,0.498,0.208}
  \definecolor{cupblue}{rgb}{0,0,.5}
  \definecolor{cupred}{rgb}{1,0.04,0}
  \definecolor{cuppink}{rgb}{0.925,0,0.545}
  \definecolor{cupmagenta}{rgb}{0.624,0.161,0.424}
  \definecolor{cupbrown}{rgb}{0.71,0.212,0.133}

  \definecolor{cupgreen}{rgb}{0,0,0}
  \definecolor{cupblue}{rgb}{0,0,0}
  \definecolor{cupred}{rgb}{0,0,0}
  \definecolor{cuppink}{rgb}{0,0,0}
  \definecolor{cupmagenta}{rgb}{0,0,0}
  \definecolor{cupbrown}{rgb}{0,0,0}

\definecolor{TITLE}{rgb}{0,0,0}
\definecolor{AUTHOR1}{rgb}{0.00,0.59,0.00}
\definecolor{AUTHOR2}{rgb}{0.50,0.00,1.00}
\definecolor{SECTION}{rgb}{0.50,0.00,1.00}
\definecolor{FOOTTITLE}{rgb}{0.00,0.50,0.75}
\definecolor{THM}{rgb}{0.8,0,0.1}
\definecolor{SEC}{rgb}{0,0,1}

\makeatletter
\normalsize
\newtheorem{theorem}{{\color{THM} Theorem}}[section]
\newtheorem{proposition}[theorem]{{\color{THM}Proposition}}
\newtheorem{corollary}[theorem]{{\color{THM}Corollary}}
\theoremstyle{definition}

\newtheorem{definition}[theorem]{{\color{THM}Definition\ }}

\newtheorem{example}[theorem]{{\color{THM}Example}}

\numberwithin{equation}{section}

\usepackage{amscd}
\begin{document}
\title{Mixed  $r$-stirling numbers of the second kind}
\author{ Daniel Yaqubi, Madjid Mirzavaziri and Yasin Saeednezhad}
\address{Department of Pure Mathematics, Ferdowsi University of Mashhad, P. O. Box 1159,
Mashhad 91775, Iran.}
\email{mirzavaziri@um.ac.ir, mirzavaziri@gmail.com}
\email{daniel\_yaqubi@yahoo.es}
\begin{abstract}
The Stirling number of the second kind $n\brace k$ counts the number of ways to partition a set of $n$ labeled balls into $k$ non-empty unlabeled cells. We  extend  this problem and give a new statement of the $r$-Stirling numbers of the second kind and $r$-Bell numbers. We also introduce the \textit{$r$-mixed Stirling number of the second kind and $r$-mixed Bell numbers}. As an application of our results we obtain a formula for the number of ways to write an integer $m > 0$ in the form $m_1\cdot m_2\cdot\ldots\cdot m_k$, where $k\geqslant 1$ and $m_i$'s are positive integers greater than $1$.
\end{abstract}

\maketitle

\section{Introduction}
Stirling numbers of the second kind, denoted by $n\brace k$, is the number of partitions of a set with $n$ distinct elements into $k$ disjoint non-empty sets. The recurrence relation
\[{n\brace k}={n-1\brace k-1}+k{n-1\brace k},\]
is valid for $k>0$, but we also require the definitions
 \[{0\brace 0}=1,\]
 and
 \[{n\brace 0}={0\brace n}=0,\]
 For $n>0$.
As an alternative definition, we can say that the $n\brace k$'s are the unique numbers satisfying
\begin{align}\label{bino}
x^{n}=\sum_{k=0}^n {n\brace k} x(x-1)(x-2)\ldots(x-k+1).
\end{align}
An introduction on Stirling numbers can be found in \cite{Cha}. Bell numbers, denoted by $B_n$, is the number of all partitions of a set with $n$ distinct elements into disjoint non-empty sets. Thus
$B_n=\sum_{k=1}^n {n\brace k}$.
These numbers also satisfy the recurrence relation $B_{n+1}=\sum_{k=0}^n{n\choose k}B_{k}$. See \cite{Bra,  Mez} .

Values of $B_n$ are given in Sloane's on-line Encyclopaedia of Integer Sequences \cite{Slo} as the sequence A000110. The sequence A008277 also gives the triangle of Stirling numbers of the second kind. There are a number of well-known results associated with them in \cite{Bra, Tuf, Cha}.

In this paper we consider the following  new problem.
\begin{itemize}
\item [$\star$]{Consider $b_1+b_2+\ldots+b_n$ balls with $b_1$ balls labeled $1$, $b_2$ balls labeled $2$, $\ldots$, $b_n$ balls  labeled $n$ and $c_1+c_2+\ldots+c_k$ cells with $c_1$ cells labeled $1$, $c_2$ cells labeled $2$, $\ldots$, $c_k$ cells labeled $k$. Evaluate the number of ways to partition the set of these balls into cells of these types.}
\end{itemize}
In the present paper, we just consider the following two special cases of the mentioned problem: $b_1=\ldots=b_n=1, c_1,\ldots,c_k\in\mathbb{N}$ and $b_1,\ldots,b_n\in\mathbb{N}, c_1=\ldots=c_k=1$.

As an application of the above problem, for a positive integer $m$ we evaluate the number of ways to write $m$ in the form $m_1m_2\ldots m_k$, where $k\geqslant 1$ and $m_i$'s are positive integers.

 \section{Mixed Bell and Stirling Numbers of second kinds }
 \begin{definition}
 A multiset is a pair $(A,m)$ where $A$ is a set and $m:A\to\mathbb{N}$ is a function.
 The set $A$ is called the set of underlying elements of $(A,m)$. For each $a\in A$, $m(a)$ is called the multiplicity of $a$.
\end{definition}
A formal definition for a multiset can be found in \cite{Aig}. Let $A=\{1,2,\ldots,n\}$ and $m(i)=b_i$ for $i=1,2,\ldots,n$. We denote the multiset $(A,m)$ by $\mathcal{A}(b_1,\ldots,b_n)$. Under this notation, the problem of partitioning $b_1+b_2+\ldots+b_n$ balls with $b_1$ balls labeled $1$, $b_2$ balls labeled $2$, $\ldots$, $b_n$ balls  labeled $n$ into $c_1+c_2+\ldots+c_k$ cells with $c_1$ cells labeled $1$, $c_2$ cells labeled $2$, $\ldots$, $c_k$ cells labeled $k$ can be stated as follows.
 \begin{definition}
Let $\mathcal{B}=\mathcal{A}(b_1,\ldots,b_n)$ and $\mathcal{C}=\mathcal{A}(c_1,\ldots,c_k)$. Then the number of ways to partition $\mathcal{B}$ balls into  $\mathcal{C}$ non-empty cells is denoted by ${\mathcal{B}\brace\mathcal{C}}$. These numbers are called the \textit{mixed partition numbers }. If cells are allowed to be empty, then we denote the number of ways to partition these balls into these cells by ${\mathcal{B}\brace\mathcal{C}}_0$.
\end{definition}
The following result is straightforward.
\begin{proposition}
Let $\mathcal{B}=\mathcal{A}(b_1,\ldots,b_n)$ and $\mathcal{C}=\mathcal{A}(c_1,\ldots,c_k)$. Then
\[{\mathcal{B}\brace\mathcal{C}}_0=\sum_{1\leqslant i\leqslant k, 0\leqslant  j_i\leqslant c_i}{\mathcal{B}\brace\mathcal{J}},\]
where $\mathcal{J}=\mathcal{A}(j_1,\ldots,j_k)$.
\end{proposition}
\section{Two Special Cases}
In this section we consider two special cases. The first case is $b_1=b_2=\ldots=b_n=1$ and $c_1,\ldots,c_k\in\mathbb{N}$ and the second case is $b_1,\ldots,b_n\in\mathbb{N}$ and $c_1=\ldots=c_k=1$.

Note that if $b_1=b_2=\ldots=b_n=1$ and $c_1=k, c_2=\ldots=c_k=0$ then ${\mathcal{B}\brace\mathcal{C}}={n\brace k}$ and ${\mathcal{B}\brace\mathcal{C}}_0=\sum_{i=1}^k{n\brace i}$. We denote $\sum_{i=1}^k{n\brace i}$ by ${n\brace k}_0$. Moreover, if $b_1=b_2=\ldots=b_n=1$ and $c_1=n, c_2=\ldots=c_k=0$ then ${\mathcal{B}\brace\mathcal{C}}_0=B_n$.
\begin{definition}
 Let $n,k$ and $r$ be positive integers, $b_1=b_2=\ldots=b_n=1$ and $c_1=r, c_2=\ldots=c_k=1$. Then we denote ${\mathcal{B}\brace\mathcal{C}}$ by $S(n,k,r)$. These numbers are called the \textit{mixed Stirling numbers of the second kind}. In this case ${\mathcal{B}\brace\mathcal{C}}_0$ is also denoted by $B_0(n,k,r)$ and is called \textit{mixed Bell numbers}.
  \end{definition}
Let us illustrate this definition by an example.
\begin{example}
We evaluate $B_0(2,2,2)$. Suppose that our balls are $\circledast$ and $\circledcirc$, and our cells are $(~),(~)$ and $[~]$. The partitions are
\begin{center}
\begin{tabular}{ccc}
$(\circledast\circledcirc)$&(\quad)&[\quad],\\
$(\circledast)$&$(\circledcirc)$&[\quad],\\
$(\circledast)$&(\quad)&$[\circledcirc]$,\\
$(\circledcirc)$&(\quad)&$[\circledast]$,\\
(\quad)&(\quad)&$[\circledast\circledcirc]$.
\end{tabular}
\end{center}
Thus $B_0(2,2,2)=5$.
\end{example}
\begin{proposition}\label{BB}
  Let $n,k$ and $r$ be positive integers and ${0\brace k}_0=1$. Then
\[B_0(n,k,r)=\sum_{\ell=0}^n{n\choose \ell}{\ell\brace r}_0(k-1)^{n-\ell}.\]
 \end{proposition}
 \begin{proof}
 Choose $\ell$ balls in $n\choose \ell$ ways and put them in $r$ cells in ${\ell\brace r}_0$ ways. We then have $n-\ell$ different balls and $k-1$ different cells. Each of these balls has $k-1$ choices. Thus the number of ways to put the remaining balls into the remaining cells is $(k-1)^{n-\ell}$.
 \end{proof}
 \begin{proposition}\label{BBB}
  Let $n,k$ and $r$ be positive integers. Then
\[S(n,k,r)=\sum_{\ell=r}^{n-k+1}{n\choose \ell}{\ell\brace r}{n-\ell\brace k-1}(k-1)!.\]
 \end{proposition}
 \begin{proof}
 Choose $\ell\geqslant r$ balls in $n\choose \ell$ ways and put them in $r$ non-empty cells in ${\ell\brace r}$ ways. We then have $n-\ell$ different balls and $k-1$ different cells. Thus the number of ways to put the remaining balls into the remaining cells is $(k-1)!{n-\ell\brace k-1}$. Note that we should have $k-1\leqslant n-\ell$.
 \end{proof}
\begin{proposition}
Let $n,k$ and $r$ be positive integers. Then
\[S(n,k,r)=\sum_{0\leqslant s\leqslant r, 0\leqslant t\leqslant k-1}(-1)^{t+\varepsilon_s}{k-1\choose t}B_0(n,k-t,r-s),\]
where
\[\varepsilon_s=\left\{\begin{array}{ll}
0 & \mbox{if~} s=0\\
1 & \mbox{otherwise~}
\end{array}\right. \]
\end{proposition}
\begin{proof}
Let
\begin{eqnarray*}
E_i&=&\mbox{the set of all partitions in which~} i \mbox{~cells labeled 1 are empty,~}\quad i=1,\ldots,r;\\
F_j&=&\mbox{the set of all partitions in which the cell labeled~} j \mbox{~is empty,~}\quad j=2,\ldots,k.
\end{eqnarray*}
Then $S(n,k,r)=B_0(n,k,r)-|(\cup_{i=1}^r E_i)\cup(\cup_{j=2}^k F_j)|+|\cup_{i=1}^r E_i|+|\cup_{j=2}^k F_j|$. We have
\begin{eqnarray*}
|E_s\cap F_{j_1}\cap\ldots\cap F_{j_t}|&=&B_0(n,k-t,r-s),\quad 1\leqslant s\leqslant r, 1\leqslant t\leqslant k-1,\\
|E_s|&=&B_0(n,k,r-s),\quad 1\leqslant s\leqslant r,\\
|F_{j_1}\cap\ldots\cap F_{j_t}|&=&B_0(n,k-t,r),\quad 1\leqslant t\leqslant k-1.
\end{eqnarray*}
Now the Inclusion Exclusion Principle implies the result.
\end{proof}
\begin{proposition}\label{bioo}
Let $n,k$ and $r$ be positive integers. Then
\[ S(n,k,r)=S(n-1,k,r-1) + (k-1)S(n-1,k-1,r) + (k-1+r)S(n-1,k,r).\]
\end{proposition}
 \begin{proof}
We have one ball labeled 1 and there are three cases for this ball:
\begin{itemize}
\item[] \textbf{Case I.} This ball is the only ball of a cell labeled 1. There are $S(n-1,k,r-1)$ ways for partitioning the remaining balls into the remaining cells.
\item[] \textbf{Case II.} This ball is the only ball of a cell labeled $j, 2\leqslant j\leqslant k$. There are $k-1$ ways for choosing a cell and  $S(n-1,k-1,r)$ ways for partitioning the remaining balls into the remaining cells.
\item[] \textbf{Case III.} This ball is not alone in any cell. In this case we put the remaining balls into all cells in $S(n-1,k,r)$ ways and then our ball has $k-1+r$ different choices, since all cells are now different after putting different balls into them.
\end{itemize}
\end{proof}
These results can be easily extended to the following general facts.
\begin{theorem}
Let $b_1=\ldots=b_n=1, c_1,\ldots,c_k\in\mathbb{N}, \mathcal{B}=\mathcal{A}(b_1,\ldots,b_n)$ and $\mathcal{C}=\mathcal{A}(c_1,\ldots,c_k)$. Then
\[{\mathcal{B}\brace\mathcal{C}}=\sum_{\ell_1+\ldots+\ell_k=n}\frac{n!}{\ell_1!\ldots\ell_k!}{\ell_1\brace c_1}{\ell_2\brace c_2}\ldots{\ell_k\brace c_k}.\]
\end{theorem}
\begin{theorem}
Let $b_1=\ldots=b_n=1, c_1,\ldots,c_k\in\mathbb{N}, \mathcal{B}=\mathcal{A}(b_1,\ldots,b_n)$ and $\mathcal{C}=\mathcal{A}(c_1,\ldots,c_k)$. Then
\[{\mathcal{B}\brace\mathcal{C}}=\sum_{1\leqslant i\leqslant k,0\leqslant j_i\leqslant c_i} (-1)^{\sharp(j_1,\ldots,j_k)}{\mathcal{B}\brace\mathcal{C}_{j_1,\dots,j_k}}_0,\]
where $\sharp(j_1,\ldots,j_k)$ is the number of $i$'s such that $j_i\neq0$ and ${\mathcal C}_{j_1,\ldots,j_k}={\mathcal A}(j_1,\ldots,j_k)$.
\end{theorem}
\begin{theorem}
Let $b_1=\ldots=b_n=1, c_1,\ldots,c_k\in\mathbb{N}, \mathcal{B}=\mathcal{A}(b_1,\ldots,b_n), \mathcal{C}=\mathcal{A}(c_1,\ldots,c_k), \mathcal{B}'=\mathcal{A}(b_2,\ldots,b_n)$ and $\mathcal{C}_j=\mathcal{A}(c_1,\ldots,c_{j-1},c_j-1,c_{j+1},\ldots,c_k)$. Then
\[{\mathcal{B}\brace\mathcal{C}}=(c_1+\ldots+c_k){\mathcal{B}'\brace\mathcal{C}}+\sum_{j=1}^k {\mathcal{B}'\brace\mathcal{C}_j}\]
\end{theorem}

\begin{theorem}\label{multip1}
Let $b_1,\ldots,b_n\in\mathbb{N}, c_1=\ldots=c_k=1, \mathcal{B}=\mathcal{A}(b_1,\ldots,b_n)$ and $\mathcal{C}=\mathcal{A}(c_1,\ldots,c_k)$. Then
\[{\mathcal{B}\brace\mathcal{C}}_0=\prod_{j=1}^n{b_j+k-1\choose k-1}.\]
\end{theorem}
\begin{proof}
The number of ways to partition $b_j$ balls labeled $j$ into $k$ labeled cells  is ${b_j+k-1\choose k-1} $. Now the result is obvious.
\end{proof}
\begin{theorem}\label{multip2}
Let $b_1,\ldots,b_n\in\mathbb{N}, c_1=\ldots=c_k=1, \mathcal{B}=\mathcal{A}(b_1,\ldots,b_n)$ and $\mathcal{C}=\mathcal{A}(c_1,\ldots,c_k)$. Then
\[{\mathcal{B}\brace\mathcal{C}}=\sum_{i=0}^{k-1}(-1)^i{k\choose i}\prod_{j=1}^{n}{b_j+k-i-1\choose k-i-1}.\]
\end{theorem}

\section{Application to $r$- Stirling Numbers  of the second kind and Multiplicative Partitioning }

The \textit{r-Stirling numbers of the second kind} were introduced by Broder \cite{Bor} for positive integers $n, k ,r$. The $r$-Stirling numbers of the second kinds
${n\brace k}_r$ are defined as the number of partitions of the set $\{1, 2, \ldots ,n\}$ that have $k$ non-empty disjoint subsets such that the elements
$1, 2, \ldots ,r$ are in distinct subsets. They satisfy the recurrence relations
\begin{itemize}
\item [i.]{${n \brace k}_r=0 \hspace{6.2cm} n<r$,}
\item [ii.]{${n \brace k}_r=k{n-1 \brace k}_{r-1} +{n-1 \brace k-1}_r \hspace{3cm} n>r$,}
\item [iii.]{${n \brace k}_r={n \brace k}_{r-1}-(r-1){n-1 \brace k}_{r-1} \hspace{2cm} n\geqslant r\geqslant 1$.}
\end{itemize}
The identity ~\ref{bino} extends into
\begin{align*}
(x+r)^{n}=\sum_{k=0}^n {n+r\brace k+r}_r x(x-1)(x-2)\ldots(x-k+1).
\end{align*}
The ordinary Stirling numbers of the second kind are identical to both $0$-Stirling and $1$-Stirling numbers. We give a theorem in which we express the $r$-Stirling numbers in terms of the Stirling numbers.
\begin{theorem}
For positive integers $n, k$ and $r$
\[{n \brace k}_r=\sum_{l=k-r}^{n-2r} {n-r\choose l} {l \brace k-r}{n-r+l \brace r}r!.\]
\end{theorem}
\begin{proof}
Firstly, we put the numbers of $\{1,2, \ldots ,r\}$ into $r$ cells as singletons. Now, we partitions $n-r$ elements to $k$ non-empty cells such that $r$ cells are labeled. The number of partitions of these elements is equal to $S(n-r ,r+1, k-r)$. Thus
\[{n \brace k}_r=S(n-r ,r+1, k-r)=\sum_{l=k-r}^{n-2r} {n-r\choose l} {l \brace k-r}{n-r+l \brace r}r!.\]
\end{proof}
\begin{corollary}
For positive integer $n, k$ and $r$
\[{n \brace k}_r={n-1 \brace k}_{r-1} +r{n-1 \brace k-1}_r +k{n-1 \brace k}_r. \]
\end{corollary}
\begin{proof}
It is clear by Proposition ~\ref{bioo}.
\end{proof}
\begin{definition}
Let $n, k$ and $r$ be positive integers. The \textit{$r$-mixed Stirling number of second kind} is the number of non-empty partitions of the set $\{1 ,2 , \ldots , n \}$ to
$\mathcal{C}=\mathcal{A}(c_1,\ldots,c_k)$ such that the elements $1, 2, \ldots , r$ are in distinct cells. We denote the $r$-mixed Stirling number of second kind by ${n \brace k}_{r}^\mathcal{C}$.
\end{definition}
\begin{example}
We evaluate ${4 \brace 2}_{2}^{\mathcal{A}(2,1)}$. Suppose that set of our balls is $\{1,2,3,4\}$ and our cells are $(~),(~)$ and $[~]$. The partitions are
\begin{eqnarray*}
&&(1,4)\hspace{0.1cm}(2)\hspace{0.1cm}[3],\hspace{0.3cm}
(1)\hspace{0.1cm}(2,4)\hspace{0.1cm}[3],\hspace{0.3cm}
(1)\hspace{0.1cm}(2)\hspace{0.1cm}[3,4],\hspace{0.3cm}
(1,4)\hspace{0.1cm}(3)\hspace{0.1cm}[2],\hspace{0.3cm}
(1)\hspace{0.1cm}(3,4)\hspace{0.1cm}[2],\hspace{0.3cm}\\
&&(1)\hspace{0.1cm}(3)\hspace{0.1cm}[2,4],\hspace{0.3cm}
(3,4)\hspace{0.1cm}(2)\hspace{0.1cm}[1],\hspace{0.3cm}
(3)\hspace{0.1cm}(2,4)\hspace{0.1cm}[1],\hspace{0.3cm}
(3)\hspace{0.1cm}(2)\hspace{0.1cm}[1,4],\hspace{0.3cm}
(1,3)\hspace{0.1cm}(2)\hspace{0.1cm}[4],\hspace{0.3cm}\\
&&(2,3)\hspace{0.1cm}(4)\hspace{0.1cm}[1],\hspace{0.3cm}
(1)\hspace{0.1cm}(2,4)\hspace{0.1cm}[3],\hspace{0.3cm}
(1)\hspace{0.1cm}(2,3)\hspace{0.1cm}[4],\hspace{0.3cm}
(2)\hspace{0.1cm}(4)\hspace{0.1cm}[1,3],\hspace{0.3cm}
(1)\hspace{0.1cm}(4)\hspace{0.1cm}[2,3],\hspace{0.3cm}\\
\end{eqnarray*}
so that ${4 \brace 2}_{2}^{\mathcal{A}(2,1)}=15$.
\end{example}
\begin{theorem}
Let $n, k$ and $r$ be  positive integers. If  $c_1=t ,c_2= \ldots =c_k=1$, then
\[{n \brace k}_{r}^\mathcal{C}=\sum_{i=0}^{\min\{t,r\}} \sum_{\ell=k-1+t-r}^{n-r}{n-r\choose \ell}
{r\choose i}{k-1\choose r-i}(r-i)!(n-r-\ell)^{r}S(\ell ,k-1+t-r,t-i)  .\]
\end{theorem}
\begin{proof}
We have $n$ balls with labels $1$ to $n$ and $k$ cells such that $t$ cells are unlabeled. We partitions this balls in two steps in such a way that the numbers of
 $\{1,2, \ldots ,r\}$ are put into $r$ cells as singletons.
\begin{itemize}
\item[] \textbf{Step I.} Let $i$ be the number of the $t$ cells of the first kind which contains some of the balls $1,2,\ldots,r$. Then $0\leqslant i\leqslant \min\{t,r\}$. We choose $i$ balls of the balls $1,2,\ldots,r$ and put them into $i$ cells of the $t$ cells of the first kind in ${r\choose i}$ ways. Then we choose $r-i$ cells of the $k-1$ different cells and put the remaining $r-i$ balls into them in ${k-1\choose r-i}(r-i)!$ ways. Thus if $N$ is the number of partitions of these balls into cells in this way, then
\[N=\sum_{i=1}^{r}{r\choose i}{k-1\choose r-i}(r-i)!.\]
\item[] \textbf{Step II.} Now, we have $n-r$ different balls and $t+k-1$ cells among which there are $t-i$ cells are of the first kind and the remaining cells are different. Note that we now have $k-1+t-r$ empty cells. Prior to anything, we fill these empty cells by $\ell$ balls of the $n-r$ balls. Thus $k-1+t-r\leqslant \ell\leqslant n-r$. Choose $\ell$ balls in $n-r\choose \ell$ ways and put them into cells in such a way that there are no empty set. The number of ways is $S(\ell,k-1+t-r,t-i)$, by Proposition~\ref{BBB}. Then put the remaining $n-r-\ell$ different balls into $r$ cells which contains the balls $1,2,\ldots,r$.
\end{itemize}
\end{proof}
\begin{corollary}
Let $n,k$ and $r$ be positive integers. If $ c_1,\ldots,c_k\in\mathbb{N}$ and $\mathcal{C}=\mathcal{A}(c_1,\ldots,c_k)$, then
\[{n\brace k}_{r}^\mathcal{C}=\sum_{i_1+\ldots+i_k=r}\frac{r!}{i_1!\ldots i_k!}{n-r\brace \mathcal{C}}.\]
\end{corollary}
 The \textit{ $r$-Bell numbers $B_{n,r}$} with parameters $n\geqslant r$ is the number of the partitions of a set $ \{1, 2,\ldots ,n\}$ such that the $r$ elements $1,2, \ldots,r$ are  distinct cells in each partition. Hence
\[ B_{n,r}=\sum_{k=0}^n {n+r \brace k+r}_r .\]
 It obvious that $B_n=B_{n,0}$. The name of $r$-Stirling numbers of second kind suggests the name for $r$-Bell numbers with polynomials
 \[B_{n,r}(x)=\sum _{k=0}^n {n+r \brace k+r} x^k,\]
 which is called the $r$-Bell polynomials \cite{Mih}.

 \begin{theorem}
 Let $n, k$ and $r$ be positive integers. Then
 \[B_{n,r}=\sum_{k=0}^n\sum_{\ell=0}^{n-r} {n-r\choose \ell}{\ell \brace k-r}_{0}r^{n-r-\ell} .\]
 \end{theorem}
 \begin{proof}
 Put the numbers $1, 2,\ldots ,r$ in $r$ cells as singletons. Now, we partition $n-r$ elements into $k$ cells such that $r$ cells are labeled. Thus
 \begin{eqnarray*}
 B_{n,r}&=&\sum_{k=0}^n B_{0}(n-r ,r+1, k-r)\\
&=&\sum_{k=0}^n\sum_{\ell=0}^{n-r} {n-r\choose \ell}{\ell \brace k-r}_{0}r^{n-r-\ell} .
 \end{eqnarray*}
 \end{proof}
 \begin{definition}
Let $n, k$ and $r$ be positive integers. The \textit{$r$-mixed Bell number} is the number of partitions of the set $\{1, 2, \ldots ,n \}$ to $\mathcal{C}=\mathcal{A}(c_1,\ldots,c_k)$ such that the elements $1, 2, \ldots, r$ are in distinct cells. We denote the $r$-mixed Bell number by $B_{n,r}^\mathcal{C}$.
 \end{definition}

 \begin{theorem}
Let $n, k $ and $r$ be positive integers. If $c_1=t ,c_2= \ldots =c_k=1$, then
\[B_{n,r}^\mathcal{C}=\sum_{i=1}^{r}{r\choose i} {k-i \choose r-i}(r-i)! B_{0}(n-r,k+i-1,t-i).\]
\end{theorem}
\begin{proof}
We have $n$ balls with labels $1$ to $n$ and $k$ cells such that $t$ cells are unlabeled. We partition these balls into $k$ cells such that the numbers
   $1,2, \ldots ,r$ into $r$ cells are as singletons. There are ${r\choose i}$ ways to choose element $i=1,2, \ldots , r$.
 The number of partitions of these balls is equal to
\[\sum_{i=1}^{r} {r\choose i}{k-i \choose r-i}(r-i)!.\]
  By Proposition~\ref{BB}, the number of partitions of $n-r$ labeled balls into $k+i-1$ cells such that $t-i$ cells are unlabeled  is equal to $B_{0}(n-r,k+i-1,t-i)$.
\end{proof}
\begin{corollary}
Let $n,k$ and $r$ be positive integers, then
\[B_{n,r}^\mathcal{C}=\sum_{k=r}^{n} {n \brace k}_{r}^\mathcal{C}.\]
\end{corollary}
\begin{proposition}
Let $n$ and $k$ be positive integers. If  $\mathcal{C}=\mathcal{A}(c_1,\ldots,c_k)$, then
\[B_{n,r}^\mathcal{C}= \sum_{i_1+\ldots+i_k=r}\frac{r!}{i_1!\ldots i_k!}{n-r\brace \mathcal{C}}_{0}.\]
\end{proposition}
We now give an application of Theorem~\ref{multip1} and Theorem~\ref{multip2} in multiplicative partitioning.

Let $n$ be a positive integer. A multiplicative partitioning of $n$ is a representation of $n$ as a product  of positive integers. Since the order of parts in a partition does not count, they are registered in decreasing order of magnitude.
\begin{theorem}
Let $m=p_1^{\alpha_1}\ldots p_n^{\alpha_n}$ be a positive integer, where $p_i$'s are different prime numbers. Then
\begin{itemize}
\item[i.] the number of ways to write $m$ as the form $m_1 \ldots m_k$, where $k\geqslant 1$ and $m_i$'s are positive integers is
\[\prod_{j=1}^n{\alpha_j+k-1\choose k-1};\]
\item[ii.] the number of ways to write $m$ as the form $m_1 \ldots m_k$, where $k\geqslant 1$ and $m_i$'s are positive integers  greater than $1$ is
\[\sum_{i=0}^{k-1}(-1)^i{k\choose i}\prod_{j=1}^n{\alpha_j+k-i-1\choose k-i-1}.\]
\end{itemize}
\end{theorem}

\end{document}